\documentclass[12pt,reqno, a4paper]{amsart}
\usepackage{graphicx}
\usepackage{fancyhdr}
\usepackage{enumerate}
\usepackage{amsmath}
\usepackage{amsthm}
\usepackage{mathabx}
\usepackage{amssymb}
\usepackage{relsize}
\usepackage{pdfsync}
\usepackage[colorlinks=true,linkcolor=blue,citecolor=magenta]{hyperref}

\usepackage[normalem]{ulem}

\usepackage{tikz}
\usetikzlibrary{decorations.pathmorphing,patterns}
\usetikzlibrary{calc,arrows}

\usepackage{MnSymbol}

\usepackage{xcolor}

\usepackage{changebar}
\usepackage{pdfsync}

\makeindex
\newtheorem{theorem}{Theorem}[section]
\newtheorem{lemma}[theorem]{Lemma}
\newtheorem{proposition}[theorem]{Proposition}
\newtheorem{corollary}[theorem]{Corollary}

\newcommand\bss{{\mathbf s}}

\newcommand\eps{\epsilon}

\newcommand \la{\lambda}
\newcommand \si{\sigma}

\renewcommand{\tilde}{\widetilde}
\renewcommand{\bar}{\overline}
\renewcommand{\le}{\leqslant}

\renewcommand{\ge}{\geqslant}

\newcommand{\sumtwo}[2]{\sum_{\substack{#1 \\ #2}}} 

\newcommand{\nn}{\nonumber}
\newcommand{\dis}{\displaystyle}

\newcommand{\La}{\Lambda}

\numberwithin{equation}{section}

	\definecolor{deepcerise}{rgb}{0.85, 0.2, 0.53}
	\definecolor{darkspringgreen}{rgb}{0.09, 0.45, 0.27}
\definecolor{darkpastelgreen}{rgb}{0.01, 0.75, 0.24}
\definecolor{carmine}{rgb}{0.59, 0.0, 0.09}
\definecolor{caraibbengreen}{rgb}{0.0, 0.8, 0.6}
\definecolor{greyd}{cmyk}{0,0,0,0.4}

\begin{document}

\title{ Interface fluctuations in non equilibrium stationary states: the SOS approximation}

\author{Anna De Masi}
\address{ Anna De Masi \\
Dipartimento di Ingegneria e Scienze dell'Informazione e Matematica, \\
Universit\`a degli studi dell'Aquila, L'Aquila, 67100 Italy }
\email{{\tt demasi@univaq.it}}

\author{Immacolata Merola}
\address{Titti Merola Dipartimento di Ingegneria e Scienze dell'Informazione e Matematica, \\
Universit\`a degli studi dell'Aquila, L'Aquila, 67100 Italy }
\email{{\tt immacolata.merola@univaq.it}}

\author{Stefano Olla}
\address{Stefano Olla: CNRS, CEREMADE\\
  Universit\'e Paris-Dauphine, PSL Research University\\
  75016 Paris, France }
  \email{olla@ceremade.dauphine.fr}


\begin{abstract}
We study the $2d$ stationary fluctuations of the interface in the SOS approximation of the non equilibrium stationary state found in \cite{DOP}. We prove that the interface fluctuations are of order $N^{1/4}$,  $N$ the size of the system. We also prove that the scaling limit is a stationary Ornstein-Uhlenbeck process.
\end{abstract}

\thanks{Dedicated to Joel for his important contributions to the theory of phase transition and interfaces}

\subjclass[2010]{}

\maketitle


\section{Introduction}
\label{sec:intro}

The  non equilibrium   stationary states (NESS) for diffusive systems in contact with reservoirs have been extensively studied, one of the main targets being to understand how the presence of a current affects what seen in thermal equilibrium. In particular it has been shown that fluctuations in NESS have a non local structure as opposite to what happens in thermal equilibrium. The theory  of such phenomena is well developed, \cite{mft}, \cite{DLS}  but mathematical proofs are restricted to very special systems (SEP, \cite{DLS2}, KMP, \cite{KMP}, chain of oscillators,\cite{BO} ....).  

The general structure of the NESS in the presence of phase transitions is a very difficult and open problem not only mathematically,  also a theoretical understanding is lacking. However  a breakthrough came recently from a paper by De Masi, Olla and Presutti, \cite{DOP}, where they prove that the NESS can be computed explicitly for a quite general class of Ginzburg-Landau stochastic models which include phase transitions.

The main point in \cite{DOP}  is that the NESS is still a Gibbs state but with the original hamiltonian modified by adding a slowly varying chemical potential. Thus for boundary driven Ginzburg-Landau stochastic models the analysis of the NESS is reduced to an equilibrium Gibbsian problem and, at least in principle, very fine properties of their structure can be investigated which is unthinkable for general models.

In particular we can study cases where there are phase transitions and purpose of this paper is to give an indication that the $2d$ NESS interface is much more rigid than in thermal equilibrium.

The analysis in \cite{DOP}  includes a system where the Ising model is coupled to a Ginzburgh-Landau process. In the corresponding NESS the distribution of the Ising spin is a Gibbs measure with the usual nearest neighbour ferromagnetic interaction plus a slowly varying external magnetic field.

In particular in the $2d$ square $\La_N:=[0,N]\times[-N,N]\cap\mathbb Z^2$  the NESS $\mu_N(\si)$   is 
$$\mu_N(\si)=\frac 1{Z_N} e^{-\beta H_N(\si)},\qquad \si=(\si(x)\in\{-1,1\}, x\in\La_N)$$
\begin{eqnarray*}
H_N(\si)  =  H^{\rm ising}(\si)+ \sum_{x\in\La_N} \frac{bx\cdot e_2 } {N}\,\,\si(x),\quad H^{\rm ising}(\si)=  \sumtwo{x,y\in\La_N}{  |x-y|=1} \mathbf 1_{\si(x)\ne \si(y)}\qquad e_2=(0,1)
\end{eqnarray*}
where $b>0$ is fixed by the chemical potentials at the boundaries.

We assume $\beta>\beta_c$, thus
since the
slowly varying external magnetic field $\dis{\frac{bx\cdot e_2} N}$ is positive in the half upper plane and negative  in the half lower plane,
we expect the existence of an interface, namely a connected ``open line" $\la$ in the dual lattice which goes from left to right and which separates the region with the majority of spins equal to 1 to the one with the majority of spins equal to -1.

The problem of the microscopic location of the interface has been much studied in equilibrium without external magnetic field and when the interface is determined by the boundary conditions: $+$ boundary conditions on $\La_N^c \cap \{x\cdot e_2 \ge 0\}$
and $-$ boundary conditions on $\La_N^c \cap \{x\cdot e_2 < 0\}$.

It is well known since the work initiated by Gallavotti, \cite{gal}, that in the $2d$ Ising model at thermal equilibrium  the interface fluctuates by the order of $\sqrt N$,  $N$ the size of the system.

     In this paper we argue that at low temperature (much below the critical value) and in the presence of a stationary current produced by reservoirs at the boundaries the interface is much more rigid as it fluctuates only by the order $N^{1/4}$. 
   
   We study the problem with a drastic simplification by considering the SOS approximation of the interface.
   Namely we consider the simplest case where the interface  $\la$ is a graph, namely $\la$
  is described by a function $s_x$, $x\in \{0,..,N\}$ with integers values in $\mathbb Z$.
The corresponding Ising configurations are spins  equal to -1 below $s_x$ and +1 above $s_x$. Namely 
$\si(x,i)=1$ if $i\ge s_x$ and   $\si(x,i)=-1$ if $i\le s_x$.

The interface is then made by a sequence of horizontal and vertical segments and the  Ising energy  of such configurations is $|\la|$. We normalise the energy by subtracting  the energy of the flat interface so that
the normalised energy is  
	$$\sum_{x=1}^N |s_x-s_{x-1}|=|\la|-N$$
	i.e. the sum of the  lengths of the vertical segments.

The energy due to the external magnetic field is normalised by subtracting the energy of the configuration when all $s_x$ are equal to 0. This is (below we set  $b=1$)
\begin{equation*}
2\sum_{x=0}^N  \sum_{i=1}^{|s_x|}\frac i {N} \approx \sum_{x=0}^N\frac {s_x^2}{N}
	\end{equation*}
Thus we get the  SOS Hamiltonian 
\begin{equation}
\label{ham}
 H_N(s)=\frac 1{N}\sum_{x=0}^{N}s_x^2+ \sum_{x=1}^{N}|s_x-s_{x-1}|
 \end{equation}

We prove that the stationary fluctuations of the interface in this SOS approximation  scaled by $N^{1/4}$  convergence to a stationary Ornstein-Unhlenbeck process. 

The problem addressed in this article is the behavior of the interface in the NESS and the aim is to argue that its fluctuations are more rigid than in thermal equilibrium as indicated by the SOS approximation. Thus  in the SOS approximation we prove the $N^{1/4}$ behavior in the simplest setting of Section \ref{sec:1}.

More general results similar to those in \cite{ISV} presumably apply.
We cannot use directly the results in \cite{ISV} because
their SOS models have an additional constraint
(the interface is in the upper half plane).
Our proofs have several points in common with \cite{ISV},
but since we work in a more specific setup with less constrains, they are
considerably simpler and somehow  more intuitive.

\section{Model and results}
\label{sec:1}

We consider $\La_N= \{0,...,N\}\times \mathbb Z$ and denote the configuration of the  interface with
$\bss = \{s_x\in \mathbb Z, x= 0,\dots , N\}$.
The interface increments are denoted by
$\eta_x=s_x-s_{x-1} \in \mathbb Z, x= 1, \dots, N$.

Let $\pi$ a symmetric probability distribution on $\mathbb Z$ aperiodic and
such that \\
	\begin{equation}
	\label{pi}
	\sum_{\eta\in\mathbb Z} e^{a\eta} \pi(\eta)< +\infty
\quad \forall |a|\le a_0, \text{ for some  $a_0>0$}
	\end{equation}
We denote $\sigma^2$ the variance of $\pi$ and as we shall see the result does not depend on the particular choice of $\pi$
but only on the variance $\sigma^2$.


\bigskip
For $s,  \bar s \in\mathbb Z$ define the positive kernel
	\begin{equation}
	\label{1.1a}
        T_N(s, \bar s)=e^{-\frac { s^2+\bar s^2}{2N}}  \pi(s-\bar s).
	\end{equation}
        Call $T_Nf(s)$ the integral operator with kernel $T_N$.
    $T_N$ is a symmetric positive operator in $\ell_2(\mathbb Z)$,
        and it can be checked immediately that it is Hilbert-Schmidt, consequently compact.
        Then the Krein-Rutman theorem \cite{KR} applies,
        thus there is a strictly positive eigenfunction $h_N\in \ell_2(\mathbb Z)$
        and a strictly positive eigenvalue $\la_N>0$:
        \begin{equation}
	\label{1.3}
\sum_{ s'} T_N(s, s')h_N( s')=\la_Nh_N(s), \qquad \sum_{s} h_N^2(s) = 1,
	\end{equation}
        The eigenvalue $\la_N<1$, and  $\la_N \to 1$ as $N\to\infty$, see Theorem \ref{th1}. 

\bigskip
We then observe that the Gibbs distribution  $\nu_N$  with the hamiltonian  given in \eqref{ham} and 
with the values at the boundaries distributed according to the measure $h_N(s)e^{\frac {s^2}{2N}}$  can be expressed in terms of the kernel $T_N$ and the  double-geometric distribution
$$\dis{\pi(\eta) = \frac 1V e^{-|\eta|}} \qquad V=\sum_\eta e^{-|\eta|}$$
In fact
\begin{eqnarray}
	\nn
        \nu_N(\bss) &=& \frac 1{Z_N} h(s_0)e^{\frac {s_0^2}{2N}} e^{-  \frac 1N\sum_{x=0}^N s_x^2}\prod_{x=1}^N \frac {e^{-|s_x-s_{x-1}|}}{V} h(s_N)e^{\frac {s_N^2}{2N}} 
   \\&=&     \frac 1{Z_N}  h_N(s_0)\,e^{-\frac 1{2N}\sum_{x=1}^N (s_x^2+s_{x-1}^2)} \prod_{x=1}^N \pi(\eta_x)h_N(s_N)\\&=&   \frac 1{Z_N} h_N(s_0)\prod_{x=1}^N T_N(s_{x-1}, s_x)h_N(s_N)\label{2.3a}
\end{eqnarray}
with $Z_N$ the partition function.

Call
		\begin{equation}
	\label{2.4a}
p_N(s,s'):=\frac {h_N(s')}{\la_N h_N(s)}T_N(s,s')	
\end{equation}
$p_N$ defines an irreducible positive-recurrent Markov chain on $\mathbb Z$ with reversible measure given by $h_N^2(s)$.
We call $ P_N$ the law of the Markov chain starting from the invariant measure $h_N^2(s)$.

Observe that  $\nu_N(\bss)$ in  \eqref{2.3a} is the $ P_N$- probability  of the trajectory $\bss$, indeed from \eqref{2.4a} we get
	\begin{eqnarray}
	\nn
     \nu_N(\bss) =\frac 1{Z_N} h_N(s_0)\prod_{x=1}^N T_N(s_{x-1}, s_x)h_N(s_N) =\frac {\la_N^N}{Z_N} h^2_N(s_0)\prod_{x=1}^N p_N(s_{x-1},s_x)\\
     \label{2.5a}
     \end{eqnarray}
which proves that $Z_N=\la^N_N$ and that $  \nu(\bss) = P_N(\bss)$.

We define the rescaled variables
	$$\tilde S^N(t) = \frac{s_{[tN^{1/2}]}}{N^{1/4}},\qquad t=0,1,..,N^{1/2},\quad  [] =  \text{ integer part}$$
then $\tilde S^N(t) $ is extended to $t\in[0,1]$ by linear interpolation, in this way we can consider the induced distribution $\mathcal P_N$ on the space of continuous function $C([0,1])$. We denote by $\mathcal E_N$ the expectation with respect to $\mathcal P_N$.

Our main result is the following Theorem.

\begin{theorem}\label{main}
 The process $\{\tilde S^N(t), t\in [0,1]\}$  converges in law to the stationary Ornstein-Uhlenbeck process with variance $\sigma/2$. Moreover $\dis{\lim_{N\to\infty} \la_N^{\sqrt N}= e^{-\si/2}}$.

\end{theorem}

The paper is organized as follows: in Section \ref{sec2} we give a priori estimates on the eigenfunctions $h_N$ and on the eigenvalues $\la_N$, in Section \ref{sec3} we prove convergence of the eigenfunctions $h_N$ and identify the limit, in Section \ref{sec:4} we prove Theorem \ref{main}.

\bigskip
\section{Estimates on the eigenfunctions and the eigenvalues}

\label{sec2}



\begin{theorem}
\label{th1}
The operator $T_N$ defined in \eqref{1.1a}  has a maximal positive eigenvalue $\la_N$ and a positive
 normalized eigenvector $h_N(s)\in \ell^{2}(\mathbb Z)$ as in \eqref{1.3}
with the following properties:
	\begin{enumerate}
	\item [(i)]  $h_N$ is a symmetric function.
        \item [(ii)]  $\|h_N\|_\infty\le 1$ for all $N$. 
\item [(iii)] There exists $c$ so that $\dis{1- \frac{c}{\sqrt N} \le \la_N < 1}$.
	\end{enumerate}
      \end{theorem}

\noindent {\bf Proof.}
That $h_N(s)$ is positive follows by the Krein-Rutman theorem, \cite{KR},
also $\la_N$ is not degenerate,
its eigenspace is one-dimensional.
The symmetry follows from the symmetry of $T_N$, since $h_N(-s)$ is also
eigenfunction for $\la_N$.

The $\ell_\infty$  bound follows from
\begin{equation}
  \label{eq:7}
  \|h_N\|^2_\infty = \sup_{s}\; h_N(s)^2 \le \sum_{s}\; h_N(s)^2 = 1.
\end{equation}

  
The upper bound in (iii) easily follows from
\begin{eqnarray*}
&& \hskip-1cm
\la_N \le \sum_{s,\bar s}  \pi(s- \bar s) h_N(s) h_N(\bar s)
\le     \frac 1{2} \sum_{s,\bar s}  \pi(s- \bar s) \left( h_N(s)^2 +  h_N(\bar s)^2\right) \le 1
	\end{eqnarray*}
having used that $\sum_{s} h_N(s)^2 = 1$.

To prove the lower bound in (iii) we use the variational formula 
\begin{equation}
  \label{eq:9}
  \begin{split}
    \la_N = \sup_{h} \frac{\sum_{s,s'} T_N(s, s') h(s) h(s')}{ \sum_{s} h(s)^2}
  \end{split}
\end{equation}
By
choosing $h$ with $\sum_{s} h(s)^2 = 1$, and
using the inequality
$e^{-x} \ge 1 - x$, we have a lower bound
\begin{equation}
  \label{eq:10}
  \begin{split}
  \la_N 
  \ge  \sum_{s,\bar s} \pi(s- \bar s) h(s) h(\bar s)
  - \frac 1{N} \sum_{s,\bar s} s^2\pi(s- \bar s)  h(s) h(\bar s)
\end{split}
\end{equation}
Observe that, since $\sum_{s} h(s)^2 = 1$,
  	\begin{eqnarray*}
&& \hskip-1cm
\frac 1{N} \sum_{s,\bar s} s^2 \pi(s- \bar s) h(s) h(\bar s)
\le     \frac 1{2N} \sum_{s,\bar s} s^2 \pi(s- \bar s) \left( h(s)^2 +  h(\bar s)^2\right)
\hskip.4cm\\&&
\le
     \frac 1{2N} \sum_{s} s^2  h(s)^2 +
      \frac 1{2N} \sum_{\eta,\bar s} (\bar s+\eta)^2  \pi(\eta)  h(\bar s)^2
\\&& =
\frac 1{N}\sum_{s} s^2 h(s)^2 +\frac{\sigma^2}{2N}
	\end{eqnarray*}
Thus \begin{equation}
  \label{eq:10a}
  \la_N \ge
  \sum_{s,\bar s} \pi(s-\bar s) h(s) h(\bar s) - \frac 1{N}
  \sum_{s} s^2 h(s)^2
  - \frac{\sigma^2}{2N}
  \end{equation}
 For $\alpha>0$, we choose  $h(s) = h_\alpha(s) := C_{\alpha}$  $ e^{- \alpha s^2/4}$,
with $C_\alpha =  \left(\sum_s e^{-\alpha s^2/2}\right)^{-1/2} $. Observe that for $\alpha\to 0$
\begin{equation*}
	\Big|\sqrt\alpha \sum_s e^{-\alpha s^2/2}-\int e^{-r^2/2}dr\Big|\le C\alpha \qquad
\big|\sqrt\alpha \sum_s (\alpha s^2 e^{-\alpha s^2/2}-\int r^2e^{-r^2/2}dr\Big|\le C\alpha	\end{equation*}
Thus
\begin{equation}
  \label{eq:12a}
  \sum_{s} s^2 h_\alpha(s)^2 = \alpha^{-1} + O(\alpha)\qquad \text{as}\quad \alpha\to 0.
\end{equation}
We next prove that
\begin{equation}
  \label{eq:11}
   \sum_{s,s'} \pi(s- s') h_\alpha(s) h_\alpha(s')\ \ge 1-\frac{\alpha \si^2}{4}
\end{equation}
To prove \eqref{eq:11} observe that
$h_\alpha(s) h_\alpha(s+\tau) = h_\alpha(s)^2 e^{-\alpha\tau^2/4 - \alpha s \tau/2}$, then
	\begin{eqnarray*}
&&  \sum_{s,s'} \pi(s- s')  h_\alpha(s) h_\alpha(s') =
  \sum_s  h_\alpha(s)  \sum_\tau \pi(\tau) h_\alpha(s+ \tau)
  \\&& =
  \sum_s h_\alpha(s)^2 \sum_\tau \pi(\tau) e^{-\alpha \tau^2/4} e^{-\alpha s\tau/2}
	\end{eqnarray*}
        Using again that $e^{-z} \ge 1 - z$  and the parity of $h_\alpha$ and of $\pi$
  we get	
\begin{eqnarray*}
  &&  \sum_s h_\alpha(s)^2 \sum_\tau \pi(\tau) e^{-\alpha \tau^2/4} e^{-\alpha s\tau/2}\\
   &&  \ge   \sum_s  h_\alpha(s)^2  \sum_\tau \pi(\tau)
     \left(1 - \frac{\alpha}{4} \tau^2\right)\left(1 -\frac{\alpha s\tau}{2}\right)
     = 1-\frac{\alpha \si^2}{4}
   	\end{eqnarray*}
which proves
\eqref{eq:11}.

We choose $\alpha = N^{-1/2}$ and from \eqref{eq:10a}, \eqref{eq:12a} and \eqref{eq:11} we then get
\begin{equation}
  \label{eq:13a}
  \la_N \ge 1-\frac 1{\sqrt N} \left(\frac{\si^2}4 + 
  1\right) -
\frac{\sigma^2}{2N} + O(N^{-3/2}),
\end{equation}
which gives the lower bound. \qed

\bigskip

%

Given $s$ let $s_x$ be the position at $x$ of the random walk starting at $s$, namely $\dis{s_x=s+\sum_{k=1}^x \eta_k}$
where $\{\eta_k\}_k$ are i.i.d. random variables with distribution $\pi$.
By an abuse of notation we will denote by $\pi$ also the probability distribution of the
trajectories of the corresponding random walk and by $\mathbb E_{s}$  the expectation with respect to the law of the random walk which starts from $s$.

We will use the local central limit theorem as stated in Theorem (2.1.1) in \cite{Lawler} (see in particular formula (2.5)). There exists a constant $c$ not depending on $n$ such that for any $s$:
\begin{equation}\label{25-L}
  |\pi(\sum_{k=1}^n {\eta_k=s)- \bar{p}(\sum_{k=1}^n \eta_k=s)|\le \frac{c}{n^{3/2}}}
\end{equation}
where 
$$\bar p(\sum_{k=1}^n \eta_k=s))= \frac 1{\sqrt{2\pi\si^2n}}e^{-\frac{s^2}{2\si^2n}}$$

\bigskip
By iterating \eqref{1.3}  $n$ times we get
	\begin{equation}
	\label{2.1}
 h_N( s)=\frac 1{\la_N^n} \mathbb E_{s}( e^{-\frac 1{2N}\sum_{x=0}^n s_x^2}\,\,h_N(s_n))
	\end{equation}

\begin{theorem}
  \label{teo2}
  There exist positive constants $c,C$ (independent of $N$) such that
  \begin{equation}
    \label{eq:17}
    h_N(s) \le { \frac C{N^{1/8}}}\,\, \exp\big \{-\dis{ \frac{c  s^2}{N^{1/2}}}\big\}
  \end{equation}
\end{theorem}

\noindent {\bf Proof}. Below we will write $h(s)$ for the eigenfunction $h_N(s)$, and $\lambda$ for $\lambda_N$.

Because of the symmetry of $h$, it is enough to consider $s>0$.
From \eqref{2.1} we get
	\begin{equation}
	\label{2.1b}
        h( s)\le\frac 1{\la^n} \Big[ \mathbb E_{s}( e^{-\frac 2{2N}\sum_{x=0}^n s_x^2})\Big]^{1/2}
        \Big[\mathbb E_{s}( h^2(s_n))\Big]^{1/2}
	\end{equation}
To estimate $\mathbb E_{s}( h^2(s_n))$ we use \eqref{25-L}, 
	\begin{eqnarray}
\mathbb E_{s}( h^2(s_n))&=& \sum_{s_n} \pi(\sum_{k=0}^{n}\eta_k=s_n-s)h^2(s_n) 
\nn \\
&\le& \sum_{s_n} \bar p(\sum_{k=0}^{n}\eta_k=s_n-s)h^2(s_n) + 
\frac{c}{n^{3/2}} \sum_{s_n}  h^2(s_n)
\nn \\
 &\le& \left[\frac{1}{\sqrt{2\pi n \si^2}}+ \frac{c}{n^{3/2}}\right]\sum_{s_n}  h^2(s_n)
 \nn \\
 &\le& \frac K{\sqrt n}\sum_{s'} h^2(s')= \frac K{\sqrt n}
 \label{2.1c}
	\end{eqnarray}
where $K$ is a constant  independent of $N$.

Thus for $n=\sqrt N$ we get
\begin{equation}
	\label{2.1d}
        h( s)\le
        \frac 1{\la^{\sqrt N}} \frac {\sqrt K}{N^{1/8}}
      \Big[   \mathbb E_{s}\big( e^{-\frac 1{N}\sum_{x=0}^n s_x^2}\big)\Big]^{1/2}
	\end{equation}
For $\alpha\in (0,1)$ we consider
\begin{equation}
	\label{2.2}
	z=\inf\{x: s_x   \le s(1-\alpha) \} 
      \end{equation}
  and we split the expectation on the right hand side of \eqref{2.1d}
  \begin{equation}
    \label{eq:3}
    \begin{split}
      \mathbb E_{s}\left( e^{-\frac 1{N}\sum_{x=0}^n s_x^2}\right) \le
      \mathbb E_{s}\left( e^{-\frac 1{N}\sum_{x=0}^{z-1} s_x^2} 1_{[z\le n]}\right) +
      \mathbb E_{s}\left( e^{-\frac 1{N}\sum_{x=0}^{n} s_x^2} 1_{[z > n]}\right) \\
      \le  \mathbb E_{s}\left( e^{-\frac {s^2(1-\alpha)^2}{N}z} 1_{[z\le n]}\right) +
      e^{-\frac {s^2(1-\alpha)^2(n+1)}{N}}
    \end{split}
  \end{equation}
   Calling
$M_x:= s_x-s$, and  $\La(a)=\log \mathbb E(e^{a\eta})$ for $|a|\le a_0$, see \eqref{pi}, we get that
        $ e^{aM_x-x\La(a)}$ is a martingale, so that
        \begin{equation}\label{eq:expm}
       1 = \mathbb E_{s}( e^{aM_{z\wedge n}- z\wedge n \La(a)})
            \ge   \mathbb E_{s}( e^{aM_{z}- z \La(a)} 1_{[z\le n]})
        \end{equation}
        Also
      $M_{z} \le -\alpha s$ and thus, choosing $a<0$, we have $aM_{z} \ge -a\alpha s$,
        so that:
$$
 \mathbb E(e^{-z\La(a)}1_{[z\le n]}) \le e^{a\alpha s}.
  $$
Since
$\La(a)= \frac 12 \si^2 a^2+O(a^4)$ choosing $a=-\frac {\sqrt 2 (1-\alpha)s}{\si N^{1/2}}$
we get  
\begin{equation*}
\mathbb E_s(e^{- \frac {(1-\alpha)^2s^2 }{N} z} 1_{[z\le n]}) \le
  e^{-\frac{\sqrt 2\alpha(1-\alpha) s^2}{2\si N^{1/2}}}
\end{equation*}
Recalling \eqref{eq:3}, we have
\begin{equation*}
  \mathbb E_{s}\left( e^{-\frac 1{N}\sum_{x=0}^n s_x^2}\right) \le  e^{-\frac{\sqrt 2\alpha(1-\alpha) s^2}{2\si N^{1/2}}}
  +  e^{-\frac {s^2(1-\alpha)^2(n+1)}{N}}
\end{equation*}
For $n = \sqrt N$ we thus get for there is a constant $b$ so that
\begin{equation}
  \label{eq:6}
  \Big[   \mathbb E_{s}\big( e^{-\frac 1{N}\sum_{x=0}^n s_x^2}\big)\Big]^{1/2}\le  e^{-\frac{b s^2}{N^{1/2}}}
\end{equation}
From (iii) of Theorem \ref{th1} there is $B>0$ so that $\lambda^{\sqrt N} \ge B$,
thus from \eqref{2.1d} and \eqref{eq:6} we get \eqref{eq:17}.
\qed

\bigskip

\section{Convergence and identification of the limit}
		\label{sec3}
We start the section with  a preliminary lemma.
\begin{lemma} There is $b>0$ so that
  \begin{equation}
    \label{eq:2}
    \sum_{s, \bar s} \pi(s- \bar s) \left( h_N(s) - h_N(\bar s)\right) ^2 \le \frac{b}{N^{1/2}}.
  \end{equation}
\end{lemma}

\begin{proof} Using that $\sum_sh_n(s)^2=1$ we have
  \begin{equation}\label{eq:split}
    \begin{split}
      \sum_{s, \bar s} \pi(s- \bar s) \left( h_N(s) - h_N(\bar s)\right) ^2 =
      2   \sum_{s, \bar s} \pi(s- \bar s)h^2_N(s) -2 \sum_{s, \bar s} \pi(s- \bar s)h_N(s)h_N(\bar s)
      \\ 
      =2-2\la_N -2 \sum_{s, \bar s} (1-e^{-(s^2 + \bar s^2)/2N} )\pi(s- \bar s)h_N(s)h_N(\bar s)
      \end{split}
  \end{equation}
  By  (iii) of Theorem \ref{th1} $2(1-\la_N)\le \frac{2c}{\sqrt N} $. 
By using that $1-e^x<x$  and that $\sum_s s^2 h_N(s)\le c'$, by Theorem \ref{teo2} we have
     \begin{eqnarray*}
&& 2  \sum_{s, \bar s} (1-e^{-(s^2 + \bar s^2)/2N} )\pi(s- \bar s)h_N(s)h_N(\bar s)\le
\frac 1{2N}  \sum_{s, \bar s} (s^2 + \bar s^2) \pi(s- \bar s)[h^2_N(s)+h^2_N(\bar s)]\\&& \le \frac {\si^2}{2N} + \frac {c'}{2N}
  \end{eqnarray*}
 From this \eqref {eq:2} follows.
 
\end{proof}

Define for $r\in\mathbb R$
	\begin{eqnarray}
	\label{3.1}
 \tilde h_N^2(r)=N^{1/4} h_N^2([r N^{1/4}]),\qquad [] =  \text{ integer part}
		\end{eqnarray}

\begin{proposition}
\label{teo3}
The following holds.

\begin{enumerate}

\item The sequence of measures $ \tilde h_N^2(r)dr$ in  $\mathbb R$ is tight and any limit measure is absolutely continuous with respect to the Lebesgue measure.

\item The sequence of functions $\tilde h_N(r):=N^{1/8} h_N([r N^{1/4}])$ is sequentially compact in $L^2(\mathbb R)$.

\end{enumerate}
\end{proposition}

\begin{proof}
As a straightforward consequence of Theorem \ref{teo2}, we have that
	\begin{equation}
	\label{4.4}
  \tilde h_N^2(r) \le C e^{- c\,\, r^2}
	\end{equation}
It follows that for any $\eps$
there is $k$ so that $\dis{\int_{|r|\le k} \tilde h_N^2(r)dr\ge 1-\eps}$,
which proves tightness of the sequence of probability measures
$ \tilde h_N^2(r) dr$ on $\mathbb R$. From \eqref{4.4} we also get that any limit measure
must be absolutely continuous.

To prove that the sequence $( \tilde h_N(r))_{N\ge 1}$ is sequentially compact in $L^2(\mathbb R)$ we prove below that there exists a constant $C$ such that for any $N$ and any $\delta>0$:
  \begin{equation}
    \label{eq:4}
    \int \left( \tilde h_N(r + \delta) - \tilde h_N(r) \right)^2 dr
    \le C\delta^2 
  \end{equation}

  Assume that $\pi(1) >0$, then
  \begin{equation*}
    \begin{split}
      \int \left( \tilde h_N(r + \delta) - \tilde h_N(r) \right)^2 dr =  \sum_s \left( h_N(s + [\delta N^{1/4}]) - h_N(s) \right)^2\\
      =  \sum_s \left( \sum_{i=1}^{[\delta N^{1/4}] } \left(h_N(s + i) - h_N(s + i -1)\right) \right)^2\\
      \le \frac{[\delta N^{1/4}]}{\pi(1)} \sum_s \sum_{i=1}^{[\delta N^{1/4}] } \pi(1) \left(h_N(s + i) - h_N(s + i -1) \right)^2\\
      \le \frac{[\delta N^{1/4}]^2}{\pi(1)} \sum_{s,\bar s}  \pi(s-\bar s) \left(h_N(s) - h_N(\bar s) \right)^2
      \le \frac{c\delta^2}{\pi(1)}
    \end{split}
  \end{equation*}
  The condition $\pi(1) >0$ can be relaxed easily by a slight modification of the above argument.

  From \eqref{4.4} and \eqref{eq:4},
  applying the Kolmogorov-Riesz  compactness  theorem (see e.g. \cite{olsen}),
 we get that $\tilde h_N$ is sequentially compact in $L^2(\mathbb R)$.      \end{proof}

We next identify the limit.
	\begin{proposition}
	\label{prop2}
Any limit point $u(r)$  of $\tilde h_N(r)$ in $L^2$ satisfies in weak form
	\begin{equation}
	\label{3.2}
	u(r)= \frac 1\la \mathbb E_r\Big(e^{-\frac 12\int_0^1 B^2_sds}u(B_1)\Big)
	\end{equation}
        where $B_s$ is a Brownian motion with variance $\si^2$ and with
        $B_0=r$ furthermore  $\dis{\la=\lim_{N\to\infty} \la_N^{\sqrt N}}$ which exists.

The unique solution of \eqref{3.2} (up to a multiplicative constant) is
$u(r)=\exp \{- r^2/2\sigma\}$ and $\la= e^{-\si/2}$.
\end{proposition}	

\noindent{\bf{Proof.}} Given $r$ call  $r_N =[r N^{1/4}]$, iterating \eqref{1.3} $\sqrt N$ times (assuming that $\sqrt N$ is an integer) we get
	\begin{eqnarray}
	\label{3.4}
          \tilde h_N(r) = \frac 1 {\la_N^{\sqrt N}}\,\,
          \mathbb E^N_{r_N}\Big(\exp\big\{-\frac 1{2\sqrt N} \sum_{x=0}^{\sqrt N}
          \frac {s_x^2}{N^{1/2}}\big \}\,\, \tilde h_N(N^{-1/4} s_{\sqrt N})\Big)
	\end{eqnarray}
where $\mathbb E^N_{r_N}$ is the expectation w.r.t. the random walk which starts from $r_N$.
	\begin{equation}
	\label{3.3}
s_x=r_N+\sum_{k=1}^x \eta_x,\qquad x=1,..,\sqrt N
	\end{equation}
        By the invariance principle,
        \begin{equation}
        \frac{ s_{t\sqrt N} - r_N}{N^{1/4}} \ \longrightarrow \ \si B_t \qquad t\in[0,1]
        \label{4.5}
        \end{equation}
        in law, where $B_t$ is a standard Brownian motion which starts from $0$.

        Take a subsequence along which $\tilde h_N$ converges strongly in  $L^2(\mathbb R)$ and call  $u(r)$ the limit point. 
        Choosing a test function $\varphi\in L^2(\mathbb R)$, and denoting
        $\pi_n(s) = \pi\left(\sum_{k=1}^{n} \eta_k = s\right) $, we get along that sequence
        \begin{equation}
        \label{4.10}
        \begin{split}
N^{-1/4} \sum_{s'}\varphi(N^{-1/4} s') \mathbb E^N_{s'}\Big(\exp\big\{-\frac 1{2\sqrt N}
 \sum_{x=0}^{\sqrt N} \big(\frac {s_x}{N^{1/4}})^2\big \}\, \big
 |\tilde h_N(N^{-1/4} s_{\sqrt N}) -u(N^{-1/4} s_{\sqrt N}\big |\Big) \\
 \le N^{-1/4} \sum_{s'}\varphi(N^{-1/4} s') \mathbb E^N_{s'}
 \Big( \big|\tilde h_N(N^{-1/4} s_{\sqrt N}) -u(N^{-1/4} s_{\sqrt N}\big |\Big)\\
= N^{-1/4} \sum_{s,s'}\varphi(N^{-1/4} s') \pi_{[\sqrt N]}\left( s-s'\right)
 \Big( \big|\tilde h_N(N^{-1/4} s) -u(N^{-1/4} s) \big |\Big)
 \\
 \le  N^{-1/4} \sum_{s'} |\varphi(N^{-1/4} s')| 
 \left(\sum_s \pi_{[\sqrt N]}\left( s-s'\right) \big|\tilde h_N(N^{-1/4} s) -u(N^{-1/4} s)\big |^2\right)^{1/2}\\
 \le \left(N^{-1/4} \sum_{s'} |\varphi(N^{-1/4} s')|^2\right)^{1/2}
 \left(N^{-1/4} \sum_{s'} \sum_s \pi_{[\sqrt N]}\left( s-s'\right)
   \big|\tilde h_N(N^{-1/4} s) -u(N^{-1/4} s)\big |^2\right)^{1/2}\\
 = \left(N^{-1/4} \sum_{s'} |\varphi(N^{-1/4} s')|^2\right)^{1/2}
 \left(N^{-1/4}  \sum_s \big|\tilde h_N(N^{-1/4} s) -u(N^{-1/4} s)\big |^2\right)^{1/2}\\
 \le C \|\varphi\|_{L^2} \|\tilde h_N - u\|_{L^2} \ \mathop{\longrightarrow}_{N\to \infty} \ 0.
\end{split}
\end{equation}

Since the exponential on the right hand side of \eqref{3.4} is a bounded functional of the random walk, from \eqref{4.5} we get (along the chosen sequence),
\begin{eqnarray}
\nn
&&\hskip-2cm  \lim_{N\to \infty}\mathbb E^N_{r_N}\Big(\exp\big\{-\frac 1{2\sqrt N}
  \sum_{x=0}^{\sqrt N} \big(\frac {s_x}{N^{1/4}})^2\big \}\, u(N^{-1/4} s_{\sqrt N}) \Big)
	\\&&\nn\hskip.5cm
=\lim_{N\to \infty}\mathbb E^N_{0}\Big(\exp\big\{-\frac 1{2\sqrt N}
  \sum_{x=0}^{\sqrt N} \big(\frac {s_x+r_N}{N^{1/4}})^2\big \}\,  u(N^{-1/4} s_{\sqrt N}) \Big)
  \\&&\hskip.5cm \label{4.11}
  =\mathbb E_0 \Big(e^{-\frac {1}2 \int_0^1 (\si B_s + r)^2 ds}u(\sigma B_1+ r)\Big)
	\end{eqnarray}
        where $\mathbb E_0 $ is the expectation w.r.t.
        the law of a standard Brownian motion starting at 0
        and the limits are intended in the weak $L^2$ sense.

        Since $\tilde h_N$ is converging strongly in $L^2$ (along the subsequence we have chosen)
        and the expectation on the right hand side of \eqref{3.4} has a finite limit,
        we get that the limit of $\la_N^{\sqrt N}$ must exists.

\medskip

Observe that for a standard Brownian motion $\{B_s\}_{s\in[0,1]}$ we have that
	\begin{eqnarray*}
\exp\big\{ -\frac{1}2 \int_0^t (\sigma B_s + r)^2 ds - \int_0^t (\sigma B_s+ r) dB_s\big\},\qquad \text{ is a martingale.}
	\end{eqnarray*}
Furthermore by Ito's formula
	\begin{equation*}
- \sigma\int_0^1 (\sigma B_s + r) dB_s= -\frac {1}2 (\si B_1+r)^2 + \frac {r^2}2 + \frac{\sigma^2}2
	\end{equation*}
        Thus
        \begin{equation*}
          \begin{split}
            1 = \mathbb E\left( \exp\big\{ -\frac{1}2 \int_0^1 (\sigma B_s + r)^2 ds
              - \int_0^1 (\sigma B_s+ r) dB_s\big\}\right)\\
            = \mathbb E\left(\exp\big\{ -\frac{1}2 \int_0^t (\sigma B_s + r)^2 ds
              -\frac {1}{2\sigma} (\si B_1+r)^2
              + \frac {r^2}{2\sigma} + \frac{\si}2\big\}\right)
        \end{split}
      \end{equation*}
      that implies
	\begin{equation*}
e^{-\frac{r^2}{2\sigma}}
=e^{\si/2} \mathbb E\Big(e^{-\int_0^1 (\si B_s + r)^2 ds}e^{-\frac 1{2\si} (\si B_1 +r)^2}\Big)\end{equation*}
Comparing with \eqref{3.2} we identify $u(r)$ and $\la$.\qed

\medskip

We thus have the following corollary of Proposition \ref{teo3} and Proposition \ref{prop2}.
\begin{corollary}
\label{coro}
The sequence of measures $ \tilde h_N^2(r)dr$ in  $\mathbb R$ converges weakly to the gaussian measure $g^2(r)dr$ where
$g(r)= (\pi\si)^{-1/4} e^{-r^2/2\si}$.

Moreover for any $\psi, \varphi\in C_b(\mathbb R)$ and  any $t\in[0,1]$
	\begin{eqnarray}
\nn &&\hskip-1cm
\lim_{N\to\infty} \frac 1 {\la_N^{\sqrt N}}\,\,\frac 1{N^{1/4}}\sum_{s} \tilde h_N(N^{-1/4}s)\psi(N^{-1/4}s)
\\&&\nn\hskip1cm\mathbb E^N_{s}\Big(\exp\big\{-\frac 1{2\sqrt N} \sum_{x=0}^{[t\sqrt N]} \frac {s_x^2}{N^{1/2}}\big \}\,\, \tilde h_N(N^{-1/4} s_{[t\sqrt N]})\varphi(N^{-1/4} s_{[t\sqrt N]})\Big)
          \\&&= e^{\si/2}\int\psi(r)g(r) \mathbb E_r \Big(e^{-\frac {1}2 \int_0^t (\si B_s)^2 ds}\varphi(\sigma B_t)g(\sigma B_t)\Big)dr
          \label{4.12}
	\end{eqnarray}
where $\mathbb E_r $ is the expectation w.r.t. the law of the Brownian motion starting at $r$.

\end{corollary}

\begin{proof} From Proposition \ref{prop2} we have that any subsequence of $ \tilde h_N(r)$ converges in $L^2(\mathbb R)$ to $c e^{-r^2/2\si}$ but since $\| \tilde h^2_N\|_{L^2}=1$ we get that $c$ must be equal to $ (\pi\si)^{-1/4} $. This together with (1) of  Proposition \ref{teo3} concludes the proof.

The proof of \eqref{4.12} is an adaptation  of \eqref{4.10} and \eqref{4.11}.\end{proof}

\section{ Proof of Theorem \ref{main}}
\label{sec:4}

Recall that   $\mathcal P_N$  and $\mathcal E_N$ denote respectively  the law and the expectation  in $C([0,1])$ of the process $ \tilde S_N(t) =N^{-1/4} s_{[tN^{1/2}]} $  induced by the law of the Markov chain with transition probabilities given in \eqref{2.4a} and initial distribution the invariant measure $\tilde h^2_N(r)dr$. 

\begin{proposition}
  \label{inv-princ}
  The finite dimensional distributions of $\tilde S_N(t)$, $t\in [0,1]$, 
  converge in law to those of the stationary Ornstein-Uhlenbeck.
\end{proposition}

\begin{proof}
 For any $k$, any $0\le \tau_1<..<\tau_k\le 1$ and any collection of continuous bounded functions with compact support $\varphi_0,\varphi_1$, ..$\varphi_k$, setting $t_i=\tau_i-\tau_{i-1}$, $i=1,..,k$, $\tau_0=0$ we have
	\begin{eqnarray*}
&&\hskip-1cm
\mathcal E_N\Big(\varphi_0(\tilde S_N(0))\varphi_1(\tilde S_N(t_1))...\varphi_k(\tilde S_N(t_k))\Big)=N^{-1/4}\sum_{r_0\in N^{-{1/4}}\mathbb Z} \tilde h_N(r_0)\varphi(r_0)\la_N^{-k\sqrt N}
\\&&\hskip.6cm
\mathbb E_{r_0}^N\Bigg(e^{-\frac 1{2\sqrt N} \sum_{x=0}^{[t_1\sqrt N]} \frac {s_x^2}{N^{1/2}}}\,\,\varphi_1(r_1)\mathbb E_{r_1}^N\Big(e^{-\frac 1{2\sqrt N} \sum_{x=0}^{[t_2\sqrt N]} \frac {s_x^2}{N^{1/2}}}\,\,\varphi_2(r_2)
\\&&\hskip2.6cm
....\mathbb E_{r_{k-1}}^N\Big(e^{-\frac 1{2\sqrt N} \sum_{x=0}^{[t_k\sqrt N]} \frac {s_x^2}{N^{1/2}}}\,\,\tilde h_N(r_k)\varphi_k(r_{k})\Big)..\Big)\Bigg)
	\end{eqnarray*}
where $r_i=N^{-1/4}\big[r_{i-1}+\sum_{x=1}^{[t_i\sqrt N]} \eta_x]$. Then from a ripetute use of \eqref{4.12} we get
\begin{eqnarray}
	 \nn
&&\hskip-.8cm \lim_{N\to\infty} \mathcal E_N\Big(\varphi_0(\tilde S_N(0))\varphi_1(\tilde S_N(t_1))\varphi_2(\tilde S_N(t_2))...\varphi_k(\tilde S_N(t_k))\Big)
\\&&=e^{k\si/2}\int g(r_0)\varphi(r_0)\mathbb E_{r_0}\Big(e^{-\int_0^{t_1} \si B_s}\varphi_1(\si B_{t_1})
...e^{-\int_0^{t_k} \si B_s}\varphi_k(B_{t_k})g(B_{t_k})\Big)dr_0
\nn \end{eqnarray}
 \end{proof}

To conclude the proof of Theorem \ref{main} we need to show tightness of $\mathcal P_N$ in $C([0,1])$; this is a consequence of Proposition \ref{tig} below, 
see Theorem 12.3, eq. (12.51) of \cite{billingsley}.
\begin{proposition} 
\label{tig} There is $C$ so that for all $N$,
   \begin{equation}
    \label{eq:5}
 \mathcal E_N\left( \left(\tilde S_N(t) - \tilde S_N(0)\right)^4\right) \le Ct^{3/2}.
  \end{equation}
\end{proposition}

\begin{proof} 
  \begin{eqnarray}\label{eq:4mom}
&&
\mathcal E_N\left( \left(\tilde S_N(t) - \tilde S_N(0)\right)^4\right)\nonumber \\
    &&=  \la_N^{-\sqrt N}    \sum_{s} h_N(s)
       \mathbb E^N_{s}\left( e^{-\frac 1{2\sqrt N} \sum_{x=0}^{[N^{1/2}t]} \frac {s_x^2}{N^{1/2}}}
        \left(\tilde S_N(t) - s N^{-1/4}\right)^4 h_N(s_{[N^{1/2}t]})\right) \nonumber
         \\ &&
              \le   \la_N^{-\sqrt N}
              \sum_{s} h_N(s) \mathbb E^N_{s}\left(\left(\tilde S_N(t) - s N^{-1/4}\right)^4
               h_N(s_{[N^{1/2}t]})\right) \nonumber
                \\&&
         \le C \la_N^{-\sqrt N}    \sum_{s,s'} h_N(s) h_N(s')
         \pi_{[N^{1/2}t]} (s-s') 
         \left|\frac{s-s'}{t^{1/2}N^{1/4}}\right|^4 t^2
  \end{eqnarray}
  where  $\pi_n(s) = \pi\left(\sum_{k=1}^{n} \eta_k = s\right) $. By Proposition 2.4.6 in \cite{Lawler}, if $\pi$ is aperiodic with finite 4th moments, as in our case, we have the bound
  \begin{equation}
    \label{eq:12}
    \pi_n(s) \le \frac{C}{n^{1/2}} \left(\frac{\sqrt{n}}{|s|}\right)^4, \qquad \forall s\in\mathbb Z.
  \end{equation}
 From this estimate it follows that the right hand side of \eqref{eq:4mom} is bounded by
  \begin{eqnarray*}
        \\&& \le  t^2 C' \la_N^{-\sqrt N}    \sum_{s,s'} h_N(s) h_N(s')   \frac{1}{\sqrt{ t N^{1/2}}}
         = C' t^{3/2} \la_N^{-\sqrt N} N^{-1/4}  \left(\sum_{s} h_N(s)\right)^2,
  \end{eqnarray*}
   By \eqref{eq:17} we have that $\sum_{s} h_N(s) \le N^{1/8}$, and the bound follows.
\end{proof}

 {\bf Acknowledgments.}
We thank S. Shlosman for helpful discussions. 
A.DM thanks very warm hospitality  at the University of Paris-Dauphine where part of this work was performed.
This work was partially supported by ANR-15-CE40-0020-01 grant LSD.


\begin{thebibliography}{99}

\bibitem{mft} L. Bertini, A. De Sole, D. Gabrielli, G. Jona-Lasinio, C. Landim (2015) \emph{Macroscopic Fluctiation Theory}, Rev. Mod. Phys. 87, 593.

\bibitem{BO} C. Bernardin, S. Olla (2005) {\em Fourier law for a microscopic model of heat conduction} Journ. Stat. Phys., {\bf 121}, 271-289

\bibitem{billingsley} P. Billingsley {\em Convergence of Probability measures} Wiley, New York, 1968

\bibitem{DOP} A. De Masi, S. Olla, E. Presutti, (2019) {\em A note on Fick's law with phase transitions } Journal Statistical Physics, {\bf 175}, 203-21.


\bibitem{DLS} B. Derrida, J. L. Lebowitz, and E. R. Speer (2001) {\em Free Energy Functional for Nonequilibrium Systems: An Exactly Solvable Case}. Phys. Rev. Lett. 87, 150601.

\bibitem{DLS2} B. Derrida, J. L. Lebowitz, and E. R. Speer(2002)  {\em Large Deviation of the Density Profile in the Steady State of the Open Symmetric Simple Exclusion Process} Journ. Stat. Phys. {\bf 107}, Nos. 3/4, 599–634



\bibitem{gal} G. Gallavotti, (1972) {\em The phase separation line in the two-dimensional Ising model}, Commun. Math. Phys., 27 103-136.

\bibitem{KMP} C. Kipnis, C. Marchioro, E. Presutti, (1982) {\em Heat flow in an exactly solvable model} Journ. Stat. Phys. {\bf 27} 65–74.

\bibitem{ISV}  D. Ioffe, S. Shlosman, Y. Velenik, (2015) {\em An invariance Principle to Ferrari-Spohn diffusions} Commun. Math. Phys., {\bf 336}, 905-932

\bibitem{olsen} Harald Hanche-Olsen, Helge Holden, (2010),
  \emph{The Kolmogorov–Riesz compactness theorem},
  Expositiones Mathematicae, {\bf 28}, n. 4, 385-394,
  doi.org/10.1016/j.exmath.2010.03.001.
  
%
  
\bibitem{KR}  M.G. Krein, M.A. Rutman (1950)
  {\em Linear operators leaving invariant a cone in a Banach space},
  Ann. Math.Soc. Transl. {\bf 26}, 128



\bibitem{Lawler} G. F. Lawler, V. Limic, (2010) {\em Random walk: a modern introduction} Cambridge Studies in Advanced Mathematics 123.

\end{thebibliography}
\end{document}